\documentclass[11pt,reqno]{amsart}

\usepackage[T1]{fontenc}
\usepackage{graphics,cancel,graphicx}
\usepackage{epsfig,psfrag,manfnt}
\usepackage{amssymb,amsmath,amsfonts,amsthm,bbm,subfigure,rotating,bm,float,mathabx,wasysym,a4wide}
\usepackage{tikz}

        \makeatletter
\DeclareFontFamily{U}{tipa}{}
\DeclareFontShape{U}{tipa}{m}{n}{<->tipa10}{}
\newcommand{\arc@char}{{\usefont{U}{tipa}{m}{n}\symbol{62}}}%

\newcommand{\arc}[1]{\mathpalette\arc@arc{#1}}

\newcommand{\arc@arc}[2]{%
  \sbox0{$\m@th#1#2$}%
  \vbox{
    \hbox{\resizebox{\wd0}{\height}{\arc@char}}
    \nointerlineskip
    \box0
  }%
}
\makeatother

\usepackage{hyperref}
\hypersetup{
    colorlinks=true,
    linkcolor=blue,
    filecolor=magenta,      
    urlcolor=cyan,
}
 
\usepackage[all,cmtip]{xy}
\usepackage{amssymb,amsmath,amsfonts,amsthm,color}

\newcommand{\bba}{\mathbf{a}}

\newcommand{\bbp}{\mathbf{p}}

\newtheorem{theorem}{Theorem}[section]
\newtheorem{lemma}[theorem]{Lemma}

\newtheorem{proposition}[theorem]{Proposition}

\theoremstyle{definition}

\theoremstyle{definition}

\theoremstyle{definition}

\usepackage{graphicx,amsmath,calc}
\newlength{\mywidth}

\begin{document}
\settowidth{\mywidth}{PQRSPQ}
\settowidth{\mywidth}{PQR}
 
\title[Crux's crux's crux]{Crux's crux's crux}


\author[Amol Sasane]{Amol Sasane}
\address{Department of Mathematics \\London School of Economics\\
    Houghton Street\\ London WC2A 2AE\\ United Kingdom}
\email{A.J.Sasane@lse.ac.uk}

\begin{abstract} 
Problem 1325 from {\em Crux Mathematicorum} is revisited, and a new solution is presented. 
\end{abstract}


\maketitle 

\section{Introduction}

\noindent The following problem (proposed by Stanley Rabinowitz),
appeared as problem 1325 in {\em Crux\footnote{First `Crux' in the
    title.} Mathematicorum}.  We call this the {\em
  Crux\footnote{Second `crux' in the title.} Problem}, since the
accompanying diagram contains a shaded `cross'=crux.

\medskip 

 \noindent {\bf Crux Problem.} 
 Let $P$ be any point inside a unit circle with center
 $C$. Perpendicular chords are drawn through $P$.  Rotation of these
 chords counterclockwise about $P$ through an angle $\theta$ sweep
 out the shaded area shown in the picture below. Show that this shaded
 area only depends on $\theta$, but not on $P$ (and hence is easily
 seen to be $2\theta$ by taking $P=C$).
 \begin{figure}[H]
      \center
      \psfrag{O}[c][c]{$C$}
      \psfrag{P}[c][c]{$P$}
      \psfrag{a}[c][c]{$\bba$}
      \psfrag{Q}[c][c]{$ \widetilde{\bbp}_1(t)$}
      \includegraphics[width=4.8cm]{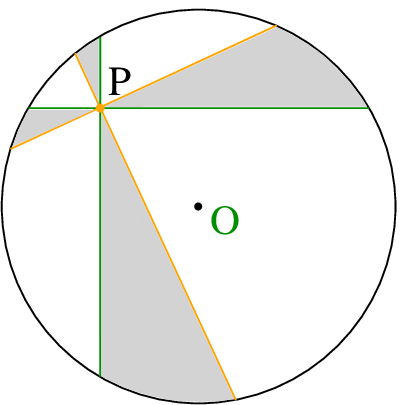}
\end{figure}

\noindent There were two solutions that appeared \cite[pp. 120-122]{C}:
\begin{itemize}
 \item[(I)] (J\"org H\"aterich) a solution using calculus and Archimedes' theorem, 
 \item[(II)] (Shiko Iwata) a non-calculus solution based on trigonometry. 
\end{itemize}
The accompanying editor's note mentioned that
Murray Klamkin generalised the problem to $n$ chords through $P$ with
equal angles of $\pi/n$ between successive chords, with the area swept
out, when these chords are rotated through an angle of $\theta$ about
$P$, then being $n\theta$. The editor's note ended with the following
parenthetical remark:

\smallskip 

 $\phantom{al}$ (This can also be proved using the solution II. Can it be proved as in solution I?) 

  \smallskip 
  
\noindent In this note, we present a calculus-based solution, based on a special case of a generalisation of ``Archimedes' theorem'', which is proved by employing 
vectors.  We purport that this solution captures in some sense the
crux\footnote{Third `crux' in the title.} of the matter.

We begin with a calculus-based proof along lines similar to the first
solution given in \cite{C}.

\section{A calculus-based proof of the Crux problem}
\label{10_feb_2020_1515_sec2}

\noindent We will use the following result. We call it Archimedes' Theorem as it is Proposition~11 
in Archimedes' work {\em The book of Lemmas} \cite[p.312]{Hea}.

\medskip 

\noindent{\bf Archimedes' Theorem.} If two mutually perpendicular
chords $A_1B_1$ and $A_2B_2$ in a unit circle with center $C$ meet at $P$, then
 $
PA_1^2+PB_1^2+PA_2^2+PB_2^2=4.
$ 
\begin{figure}[h]
      \center
      \psfrag{O}[c][c]{$C$}
      \psfrag{o}[c][c]{$P$}
      \psfrag{A}[c][c]{$B_1$}
      \psfrag{B}[c][c]{$A_1$}
      \psfrag{C}[c][c]{$B_2$}
      \psfrag{D}[c][c]{$A_2$}
      \psfrag{E}[c][c]{$D$}
      \includegraphics[width=4.8cm]{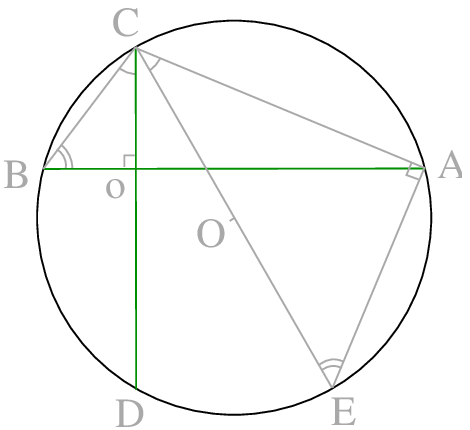}
\end{figure}
\begin{proof}$\Delta B_2PA_1\sim \Delta B_2B_1D$ since $\angle B_2 A_1 B_1=\angle B_2DB_1$ and
  $\angle B_2 PA_1=90^\circ=\angle B_2B_1D$.  So
  $\angle PB_2A_1=\angle B_1B_2D$. This implies that
  $\angle A_1CA_2=\angle B_1CD$, and so $A_1A_2=B_1D$. By Pythagoras'
  Theorem, $
PB_2^2+PB_1^2=B_2B_1^2$ and $PA_1^2+PA_2^2= A_1A_2^2$. Adding these, we obtain 
 $
PA_1^2+PB_1^2+PA_2^2+PB_2^2=B_2B_1^2+ A_1A_2^2=B_1B_2^2+ B_1D^2=B_2D^2=2^2=4.
$
\end{proof}

\noindent Now we give a calculus argument as follows. Rotating
$A_1B_1$ and $A_2 B_2$ about $P$ through an infinitesimal angle
$\textrm{d}\theta$, we obtain four sectors, with areas given by
$$
\dfrac{1}{2} PA_k^2 \;\!\textrm{d}\theta,\;\;
  \dfrac{1}{2} PB_k^2 \;\! \textrm{d}\theta, \quad k=1,2.
$$
Upon addition, and using Archimedes' Theorem, we obtain the rate of change of area 
$$
\frac{\textrm{dA}}{\textrm{d}\theta}=\frac{1}{2}(PA_1^2+PB_1^2+PA_2^2+PB_2^2)=\frac{1}{2} 4 =2, 
$$
and so the total area, if the chords are rotated through an angle
$\theta$, is given by
$$
\textrm{A}=\int_0^\theta \frac{\textrm{dA}}{\textrm{d}\theta} \;\textrm{d}\theta=\int_0^\theta 2 \;\textrm{d}\theta =2\;\!\theta.
$$

\section{A vector calculus proof}
\label{section_3}

\noindent We will first show the following:

\begin{proposition}
\label{prop_23_dec_2019_1436}
Let $P$ be any point inside a unit circle, and through $P$, let there be $n$ chords
$A_1B_1,\cdots ,A_nB_n$, such that there are equal angles of $\pi/n$ between
successive chords. Suppose moreover that $A_1B_1$ is a diameter. 
If each chord is rotated counterclockwise through
an angle $\theta$, then the total area formed by the resulting sectors
is $n\theta$.
\end{proposition}
\begin{figure}[h]
      \center
      \psfrag{P}[c][c]{${ P}$}
      \psfrag{t}[c][c]{${\scriptscriptstyle \frac{k\pi}{n}}$}
      \psfrag{a}[c][c]{${ A_1}$}
      \psfrag{b}[c][c]{${ B_1}$}
      \psfrag{A}[c][c]{${ A_{2}}$}
      \psfrag{B}[c][c]{${ B_{2}}$}
      \psfrag{c}[c][c]{${ A_n}$}
      \psfrag{d}[c][c]{${ B_{n}}$}
      \psfrag{D}[c][c]{${ \cdots}$}
      \includegraphics[width=5.4cm]{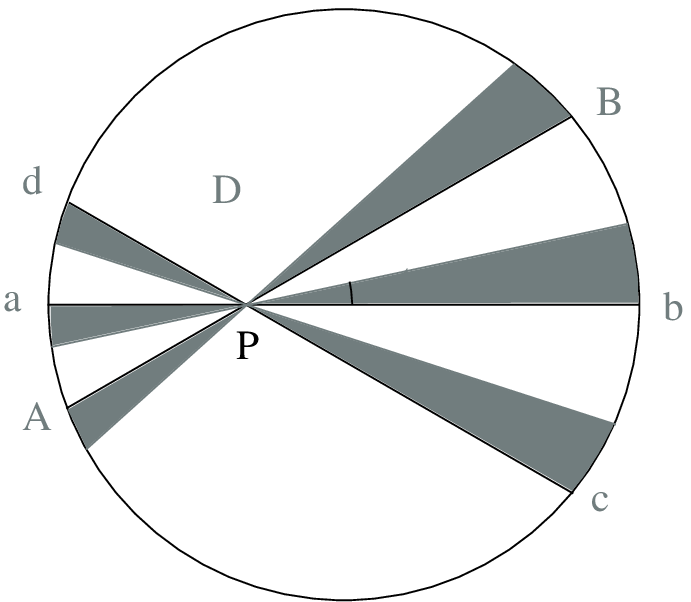}
\end{figure}

\noindent 
This will be shown to yield the generalisation (given in Theorem~\ref{10_feb_2020_14_05} below) of the Crux problem, where as
opposed to the situation above, one of the chords needn't be the diameter.

In order to prove Proposition~\ref{prop_23_dec_2019_1436},
we will first prove a special case of a generalisation of Archimedes' Theorem (Theorem~\ref{10_feb_2020_Arch_gen_14:37} in the next section, 
saying that the sum of the squared distances from a point inside a unit circle to the vertices of $n$ equally angularly spaced chords passing through that point is $2n$), 
when one of the
chords $A_1 B_1$ is the diameter.

\begin{lemma}[Generalised Archimedes' theorem $-$ special case]
\label{23_dec_2019_1840}
Let $P$ be any point inside a unit circle, and let there be $n$ chords
$A_1B_1,\cdots A_nB_n$ through $P$ such that there are equal angles of
$\pi/n$ between successive chords. Suppose, moreover, that $A_1B_1$
is a diameter. Then
 $
PA_1^2+PB_1^2+\cdots+ PA_n^2+PB_n^2=2n.
$ 
\end{lemma}
\begin{proof}
Let $C_1,\cdots, C_n$ be the centers of $A_1B_1,\cdots, A_n B_n$. 
As $A_1B_1$ is the diameter, $C_1$ is the center of the circle.  We
know that for all $1\leq k\leq n $,
\begin{eqnarray*}
 &&  \langle \overrightarrow{PA}_k- \overrightarrow{PC}_1,\overrightarrow{PA}_k- \overrightarrow{PC}_1\rangle 
=\|\overrightarrow{PA}_k- \overrightarrow{PC}_1\|_2^2=1,
 \\
  &&  \langle \overrightarrow{PB}_k- \overrightarrow{PC}_1,\overrightarrow{PB}_k- \overrightarrow{PC}_1\rangle 
=\|\overrightarrow{PB}_k- \overrightarrow{PC}_1\|_2^2=1.
\end{eqnarray*}
Expanding these, adding, and rearranging, we obtain
\begin{equation}
  \label{23_dec_2019_1829}
  \sum_{k=1}^n (PA_k^2+ PB_k^2)
=2\left\langle \overrightarrow{PC}_1,\sum_{k=1}^n (\overrightarrow{PA}_k+\overrightarrow{PB}_k)\right\rangle -2n PC_1^2+2n.
\end{equation}
So we need to determine the inner product on the RHS. We have
\begin{eqnarray*}
  \overrightarrow{PA}_k= \overrightarrow{PC}_k+\overrightarrow{C_kA}_k &\textrm{ and }&
  \overrightarrow{PB}_k= \overrightarrow{PC}_k+\overrightarrow{C_kB}_k.
\end{eqnarray*}
But since
$\overrightarrow{C_kA}_k+\overrightarrow{C_kB}_k=\mathbf{0}$, we
obtain $\overrightarrow{PA}_k+\overrightarrow{PB}_k=2\overrightarrow{PC}_k$.  
Hence
$$
\sum_{k=1}^n ( \overrightarrow{PA}_k+\overrightarrow{PB}_k)
=2\sum_{k=1}^n\overrightarrow{PC}_k =
\sum_{k=1}^n\overrightarrow{PC}_k+\sum_{k=1}^n\overrightarrow{PC}_{n-k}=\sum_{k=1}^n(\overrightarrow{PC}_k+\overrightarrow{PC}_{n-k}).
$$
\begin{figure}[!t]
      \center
      \psfrag{P}[c][c]{${ P}$}
      \psfrag{t}[c][c]{${\frac{k\pi}{n}}$}
      \psfrag{a}[c][c]{${ A_k}$}
      \psfrag{b}[c][c]{${ B_k}$}
      \psfrag{A}[c][c]{${ A_{n-k}}$}
      \psfrag{B}[c][c]{${ B_{n-k}}$}
      \psfrag{c}[c][c]{${ C_k}$}
      \psfrag{C}[c][c]{${ C_{n-k}}$}
      \psfrag{x}[c][c]{${ A_1}$}
      \psfrag{y}[c][c]{${ B_1}$}
      \psfrag{z}[c][c]{${ C_1}$}
      \includegraphics[width=5.4cm]{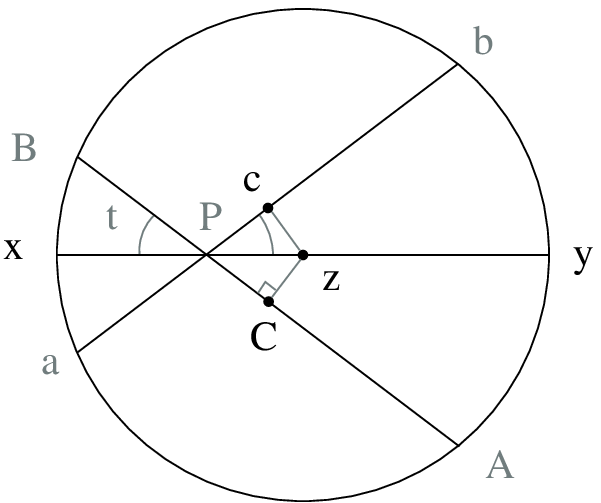}
      \caption{}
      \label{fig_11_feb_2020_12_29} 
\end{figure}
 
\noindent By referring to Figure~\ref{fig_11_feb_2020_12_29}, we see that for all
$1\leq k\leq n$,
{\small 
$$
\overrightarrow{PC}_k\!+\!\overrightarrow{PC}_{n-k}
\!=\!
2 \;\! PC_{k} \;\! \left(\cos \frac{k\pi}{n}\right)\;\! \frac{\overrightarrow{PC}_1}{PC_1}
\!=\!
2\;\!  \left(\cos \frac{k\pi}{n}\right)\;\! PC_1\;\! \left(\cos \frac{k\pi}{n}\right)\;\! \frac{\overrightarrow{PC}_1}{PC_1}
\!=\!
2\left(\cos \frac{k\pi}{n}\right)^2 \overrightarrow{PC}_1.
$$
}

\noindent 
So 
\begin{equation}
\label{23_dec_2019_1833} 
\sum_{k=1}^n ( \overrightarrow{PA}_k+\overrightarrow{PB}_k) = 
\sum_{k=1}^n ( \overrightarrow{PC}_k+\overrightarrow{PC}_{n-k}) =
\sum_{k=1}^n 2\left(\cos \frac{k\pi}{n}\right)^2  \overrightarrow{PC}_1.
\end{equation}
Now 
\begin{equation}
\label{23_dec_2019_1832} 
\sum_{k=1}^n 2\left(\cos \frac{k\pi}{n}\right)^2
=\sum_{k=1}^n \left(1+\left(\cos k\frac{2\pi}{n}\right)\right)=n+0=n,
\end{equation}
where we have used 
\begin{equation}
 \label{10_feb_14:23_2020} 
 {\textstyle \displaystyle \sum\limits_{k=1}^n\cos \left(k\frac{2\pi}{n}\right)=0}
\end{equation}
(To see \eqref{10_feb_14:23_2020}, we first note that this sum is the horizontal component of the sum $\overrightarrow{S}$ of $n$ vectors whose tails lie at the center of the unit circle 
and whose tips lie on the vertices of a regular $n$-gon. To see that $\overrightarrow{S}$ is zero, imagine 
rotating each vector counterclockwise through an angle of $\frac{2\pi}{n}$, 
and let the sum of the rotated vectors be $\overrightarrow{S'}$. On grounds of symmetry of the regular polygon, 
$\overrightarrow{S}=\overrightarrow{S'}$. On the other hand $\overrightarrow{S'}$ ought be a 
rotated version of $\overrightarrow{S}$ through an angle of $\frac{2\pi}{n}$. This can only happen if $\overrightarrow{S}=\mathbf{0}$. 
Alternative justifications of \eqref{10_feb_14:23_2020} 
can be given by first summing the geometric series 
$$
\sum_{k=1}^n e^{i\frac{2\pi}{n} k}=e^{i \frac{2\pi}{n}} \frac{1-e^{i 2\pi}}{1-e^{i\frac{2\pi}{n}}}=0,
$$
and taking real parts, or by noticing the sum of the $n$th roots of unity must add up to $0$ since the coefficient of $z^1$ in $z^n-1$ is $0$, and again taking real parts.) 

\smallskip 

\noindent 
Consequently, using \eqref{23_dec_2019_1829}, \eqref{23_dec_2019_1833},
and \eqref{23_dec_2019_1832}, we obtain
\begin{eqnarray*}
 \sum_{k=1}^n (PA_k^2+ PB_k^2)&=&2\left\langle \overrightarrow{PC}_1,\sum_{k=1}^n (\overrightarrow{PA}_k+\overrightarrow{PB}_k)\right\rangle -2n PC_1^2+2n
 \\
 &=& 2\left\langle \overrightarrow{PC}_1, n \overrightarrow{PC}_1\right\rangle -2n PC_1^2+2n
 =
 \cancel{2n PC_1^2}-\cancel{2n PC_1^2} +2n=2n.
 \qedhere
\end{eqnarray*}
\end{proof}

\noindent We are now ready to prove Proposition~\ref{prop_23_dec_2019_1436}.

\begin{proof}[Proof of Proposition~\ref{prop_23_dec_2019_1436}]
Rotating $A_1B_1,\cdots, A_n B_n$ counterclockwise about $P$ through an infinitesimal angle $\textrm{d}\theta$, we obtain $2n$ sectors, 
with areas given by
$$
  \dfrac{1}{2} PA_k^2 \;\! \textrm{d}\theta,\;\;
  \dfrac{1}{2} PB_k^2 \;\! \textrm{d}\theta, \quad k=1,\cdots, n.
$$
Upon addition, and using Lemma~\ref{23_dec_2019_1840}, we obtain that the rate of change of the total area is 
$$
\frac{\textrm{dA}}{\textrm{d}\theta}=\frac{1}{2}\sum_{k=1}^n (PA_k^2+ PB_k^2) \;\! \textrm{d}\theta =\frac{1}{2} \;\!2n =n, 
$$
and so the total area, if the chords are rotated through an angle
$\theta$, is given by
$$
\phantom{aaaaaaaaaaaaaaaaaaaaaaa} A=\int_0^\theta \frac{\textrm{dA}}{\textrm{d}\theta} \;\textrm{d}\theta=\int_0^\theta n\; \textrm{d}\theta =n\;\! \theta.
\phantom{aaaaaaaaaaaaaaaaaaaaaaa} 
\qedhere
$$
\end{proof}

\begin{theorem}
\label{10_feb_2020_14_05}
Let $P$ be any point inside a unit circle, and let there be $n$ chords
through $P$ such that there are equal angles of $\pi/n$ between
successive chords. If each chord is rotated counterclockwise through
an angle $\theta$, then the total area formed by the resulting sectors
is $n\theta$.
\end{theorem}
\begin{proof} To see how this follows from
  Proposition~\ref{prop_23_dec_2019_1436}, we first construct the
  diameter $A_1B_1$ through $P$, and consider successive anticlockwise
  rotations of this diameter through angles of $\pi/n$, resulting in
  the chords $A_2B_2,\cdots , A_n B_n$. Let the given chords from the
  theorem statement be labelled as $A_1'B_1',\cdots, A_n'B_n'$, and let their
  rotated versions (though an angle $\theta$) be labelled as
  $A_1'' B_1'',\cdots, A_n''B_n''$.
\begin{figure}[h]
      \center
      \psfrag{P}[c][c]{${\scriptstyle P}$}
      \psfrag{t}[c][c]{${\scriptstyle\theta'}$}
      \psfrag{T}[c][c]{${\scriptstyle\theta''}$}
      \psfrag{s}[c][c]{${\scriptstyle\theta}$}
      \psfrag{a}[c][c]{${\scriptstyle A_1}$}
      \psfrag{b}[c][c]{${\scriptstyle B_1}$}
      \psfrag{c}[c][c]{${\scriptstyle A_2}$}
      \psfrag{d}[c][c]{${\scriptstyle B_2}$}
      \psfrag{e}[c][c]{${\scriptstyle A_3}$}
      \psfrag{f}[c][c]{${\scriptstyle B_3}$}
      \psfrag{A}[c][c]{${\scriptstyle A_1''}$}
      \psfrag{B}[c][c]{${\scriptstyle B_1'}$}
      \psfrag{C}[c][c]{${\scriptstyle A_1'}$}
      \psfrag{D}[c][c]{${\scriptstyle B_1''}$}
      \psfrag{E}[c][c]{${\scriptstyle A_2'}$}
      \psfrag{F}[c][c]{${\scriptstyle B_2'}$}
      \psfrag{G}[c][c]{${\scriptstyle A_2''}$}
      \psfrag{H}[c][c]{${\scriptstyle B_2''}$}
      \psfrag{Z}[c][c]{${\scriptstyle C_1}$}
      \includegraphics[width=8.7cm]{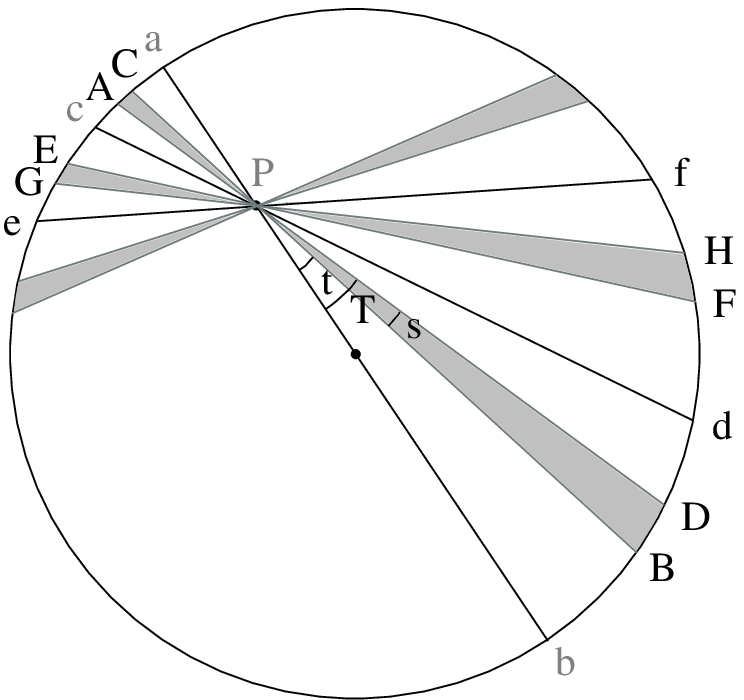}
\end{figure}

\noindent Let the angle between $A_1 B_1$ and $A_1'B_1'$ be $\theta'$,
and that between $A_1B_1$ and $A_1'' B_1''$ be $ \theta''$.  Then for
all $1\leq k\leq n$, we use the notation $ \arc{B_k' B_k''}$ for the 
sector formed by the corresponding arc with $P$, and denote the area of the sector by $\textrm{A}(\arc{B_k' B_k''})$. Then we have:
{\small 
\begin{eqnarray*}
 \sum_{k=1}^n [\textrm{A}(\arc{B_k' B_k''})+\textrm{A}(\arc{A_k' A_k''})]
 &=& 
 \sum_{k=1}^n [\textrm{A}(\arc{B_k B_k''})- \textrm{A}(\arc{B_k B_k'})+ \textrm{A}(\arc{A_k A_k''})- \textrm{A}( \arc{A_k A_k'})]\\
 &=& 
 \sum_{k=1}^n [\textrm{A}(\arc{B_k B_k''})+ \textrm{A}(\arc{A_k A_k''})]-\sum_{k=1}^n[ (\textrm{A}(\arc{B_k B_k'})+  \textrm{A}( \arc{A_k A_k'}))]\\
 &=&n\theta''-n\theta'=n(\theta''-\theta')=n\theta.\phantom{aaaaaaaaaaaaaaaaaa\sum_{k=1}^n \textrm{A}(\arc{B_k B_k''})}\hfill\qedhere
\end{eqnarray*}}
\end{proof}

\goodbreak

\section{Archimedes' Theorem}

\noindent A consequence of Theorem~\ref{10_feb_2020_14_05} is the following
generalisation of Archimedes' Theorem from the $n=2$ chord case considered earlier in Section~\ref{10_feb_2020_1515_sec2}.

\begin{theorem}[Generalised Archimedes' theorem]
\label{10_feb_2020_Arch_gen_14:37}
Let $P$ be any point inside a unit circle, and let there be $n$ 
chords $A_1B_1,\cdots A_nB_n$ through $P$ such that there are equal
angles of $\pi/n$ between successive chords. Then
 $$
PA_1^2+PB_1^2+\cdots+ PA_n^2+PB_n^2=2n.
$$ 
\end{theorem}
\begin{proof} 
By Theorem~\ref{10_feb_2020_14_05}, we know that if the chords 
are rotated through an infinitesimal angle $\textrm{d}\theta$, the sum of the
areas of the resulting sectors is $n \textrm{d}\theta$. But this area is also equal to 
 $
\frac{1}{2}(PA_1^2+PB_1^2+\cdots+ PA_n^2+PB_n^2)\;\! \textrm{d}\theta.
$ 
So we obtain $PA_1^2+PB_1^2+\cdots+ PA_n^2+PB_n^2=2n$. 
\end{proof}

\bigskip 

\noindent {\bf Acknowledgement:} The author is grateful to the reviewer for useful suggestions on improving the exposition.


\begin{thebibliography}{99}

\bibitem{C}
Solution to problem 1325, {\em Crux Mathematicorum}, volume 15, number 4, pages 120-121, April 1989. 
Available \href{https://cms.math.ca/crux/backfile/Crux_v15n04_Apr.pdf}{here}.


\bibitem{Hea}
{\em The works of {A}rchimedes}. 
Reprint of the 1897 edition and the 1912 supplement, edited by T. L. Heath, 
Dover Publications, 2002.

\end{thebibliography}
\end{document}